\documentclass{amsart}

\usepackage{NCTori}

\title{The Gauss-Manin Connection and Noncommutative Tori}
\author{Allan Yashinski}
\thanks{This research was partially supported under NSF grant DMS-1101382.}

\begin{document}

\begin{abstract} We use Getzler's Gauss-Manin connection to prove the invariance of periodic cyclic cohomology for the smooth deformation of noncommutative tori.  We explicitly calculate the parallel translation maps and use them to describe the behavior of the Chern-Connes pairing under this deformation.
\end{abstract}

\maketitle

\tableofcontents

\section{Introduction}

Our motivating problem is to understand the behavior of cyclic cohomology under deformation of the algebra structure.  Given a family of algebras $\{A_t\}_{t \in J}$ depending on a real parameter $t$, we would like to identify conditions under which the periodic cyclic cohomology $HP^\bullet(A_t)$ is independent of the parameter $t$.

Our approach to this problem uses Getzler's Gauss-Manin connection \cite{MR1261901}.  We shall work in the setting of smooth deformations $\{A_t\}_{t \in J}$, in which the algebra structures depends smoothly in some sense on the parameter $t$.  The Gauss-Manin connection was initially introduced for formal deformations, but was adapted to smooth deformations in \cite{Yashinski-SmoothDefs}.  It is a connection $\nabla^{GM}$ on the bundle of chain complexes $\{C^{\per}(A_t)\}_{t \in J}$, which exists for any smooth deformation and commutes with the coboundary map.  If the appropriate differential equations can be solved, one can produce parallel translation isomorphisms
\[ P^{\nabla^{GM}}_{s,t}: HP^\bullet(A_s) \to HP^\bullet(A_t) \]
for any $s,t \in J$.

An important feature of cyclic cohomology is that it provides numerical invariants of $K$-theory classes through the Chern-Connes pairing
\[ \langle \cdot, \cdot \rangle: HP^i(A) \times K_i(A) \to \C,\qquad i=0,1, \]
between periodic cyclic cohomology and algebraic $K$-theory.  The Gauss-Manin connection is compatible with this pairing in the following sense.  Given a smoothly varying family $\{ \phi_t \in C^{\even}(A_t) \}_{t \in J}$ of even cocycles and a smoothly varying family of idempotents $\{P_t \in M_N(A_t)\}_{t \in J}$ in matrix algebras, we have
\[ \frac{d}{dt} \langle [\phi_t], [P_t] \rangle = \langle \nabla^{GM} [\phi_t], [P_t] \rangle, \]
see \cite{Yashinski-SmoothDefs}.  A similar result holds for the pairing of odd cocycles and invertibles, representing classes in $K_1(A_t)$.  Thus being able to compute with $\nabla^{GM}$ gives insight into how the numerical invariants arising from this pairing are changing with the parameter $t$.

In this paper, our focus is on noncommutative tori, a well-studied example in noncommutative geometry \cite{MR1047281}.  The smooth noncommutative $n$-torus $\A_\Theta$ is a Fr\'{e}chet algebra which can be naturally viewed as a deformation of $C^\infty(\T^n)$, the algebra of smooth complex-valued functions on the $n$-torus.  Our main result is the integrability of the Gauss-Manin connection for the deformation of noncommutative tori.  In particular, we obtain a parallel translation isomorphism $HP^\bullet(\A_\Theta) \cong HP^\bullet(C^\infty(\T^n))$.  Since the latter can be computed in terms of de Rham homology \cite{MR823176}, this provides a deformation theoretic computation of $HP^\bullet(\A_\Theta)$.  A description of $HP^\bullet(\A_\Theta)$ is not new, as it was already computed directly by Connes in the $n=2$ case \cite{MR823176} and by Nest in the general case \cite{MR973508}.  What is novel is our method of computing it using deformation theory.

We also explicitly describe the operator $\nabla^{GM}$ in terms of a natural basis for $HP^\bullet(\A_\Theta)$.  This allows us to compute the derivatives of the Chern-Connes pairing with respect to the parameter of the deformation.  It turns out that $\nabla^{GM}$ has certain nilpotence properties with respect to the elements of this basis.  This allows us to express the Chern-Connes pairings in the noncommutative torus as a polynomial function in the parameters whose coefficients are Chern-Connes pairings for the commutative torus $C^\infty(\T^n)$.  For example, given a smoothly varying (in $\Theta$) family of idempotents $\{P_\Theta \in M_N(\A_\Theta)\}$, the Chern-Connes pairings with $P_\Theta$ are determined by, and can be computed from, the characteristic classes of the smooth vector bundle associated to $P_0 \in M_N(C^\infty(\T^n))$.  Related formulas were found for the pairing with the canonical trace by Elliott \cite{MR731772}.

Similar work on the Gauss-Manin connection was carried out independently by Yamashita \cite{Yamashita}.

\subsection*{Acknowledgements}
I'd like to thank my thesis advisor, Nigel Higson, for his suggestion of this project and helpful guidance throughout.

\section{Preliminaries}

Here we shall rapidly review the necessary background material, including the types of algebras we consider, the basics of cyclic homology, and the Gauss-Manin connection in smooth deformation theory.

\subsection{Fr\'{e}chet algebras and modules}
The topological vector spaces we work with will generally have nice properties.  They will be either nuclear Fr\'{e}chet spaces or their duals.  Many of the definitions below make sense for general spaces in $\LCTVS$, the category of complete, Hausdorff, locally convex topological vector spaces and continuous linear maps.  We shall mostly just discuss the Fr\'{e}chet case below, because some of the results become simpler to state.  See \cite{Yashinski-SmoothDefs} for more details in the general case.

By a continuous bilinear (or multilinear) map, we always mean jointly continuous, though this distinction is not necessary to make for Fr\'{e}chet spaces.  A Fr\'{e}chet algebra is a Fr\'{e}chet space equipped with a continuous associative multiplication.  Similarly, a Fr\'{e}chet module over a Fr\'{e}chet algebra has a continuous module action.

Let $R$ be a commutative unital Fr\'{e}chet algebra, which will serve as the ground ring.  The relevant examples are $R = \C$ and $R = C^\infty(J)$, the Fr\'{e}chet algebra of smooth complex-valued functions on a real interval $J$.  Given two Fr\'{e}chet $R$-modules $M$ and $N$, their \emph{projective tensor product over $R$} is a Fr\'{e}chet $R$-module $M \potimes_R N$ equipped with a universal continuous $R$-bilinear map $\iota: M \times N \to M \potimes_R N$.  That is, given any Fr\'{e}chet $R$-module $P$ and a continuous $R$-bilinear map $B: M \times N \to P$, there is a unique continuous $R$-linear map $\widehat{B}: M \potimes_R N \to P$ such that the diagram
\[ \xymatrix{
M \times N \ar[r]^-\iota \ar[rd]^-B & M \potimes_R N \ar[d]^-{\widehat{B}}\\
& P
} \]
commutes.  For more information, see \cite{MR1093462}.

By a Fr\'{e}chet $R$-algebra, we mean a Fr\'{e}chet $R$-module $A$ equipped with a continuous $R$-bilinear associative multiplication.  Notice that the multiplication induces a continuous $R$-linear map
\[ m: A \potimes_R A \to A. \]

Given a Fr\'{e}chet $R$-module $M$, let $M^\dual = \Hom_R(M, R)$ be the strong $R$-linear topological dual space, equipped with the topology of uniform convergence on bounded subsets of $M$.  Then $M^\dual$ is a complete, Hausdorff, locally convex $R$-module, but generally will not be a Fr\'{e}chet space.  We will reserve the notation $M^* = \Hom(M,\C)$ for the usual topological dual, which is the case when $R = \C$.  The transpose of a continuous $R$-linear map $F: M \to N$ is a continuous $R$-linear map
\[ F^\dual: N^\dual \to M^\dual, \qquad F^\dual(\phi) = \phi \circ F. \]

\subsection{Hochschild and cyclic homology}

See \cite{MR1600246} for more details concerning Hochschild and cyclic homology.

Given a Fr\'{e}chet $R$-algebra, let $A_+ = A \oplus R$ denote its $R$-linear unitization, which has the product
\[ (a_1, r_1)(a_2, r_2) = (a_1a_2 + r_2a_1 + r_1a_2, r_1r_2). \]
We'll perform this construction even if $A$ is already unital.  To avoid confusion, we'll denote the unit $(0,1)$ of $A_+$ by $e$.

The (continuous, $R$-linear) Hochschild chain groups of $A$ are
\[ C_0(A) = A, \qquad C_n(A) = A_+ \potimes_R A^{\potimes_R n} \quad (n \geq 1). \]
We'll also use the notation $C_\bullet^R(A)$ if we wish to emphasize the ground ring $R$.  The boundary map $b: C_n(A) \to C_{n-1}(A)$ is given on elementary tensors by
\begin{align*}
b(\widetilde{a_0} \otimes a_1 \otimes \ldots \otimes a_n) &= \sum_{j=0}^{n-1} (-1)^j \widetilde{a_0} \otimes \ldots \otimes a_ja_{j+1} \otimes \ldots \otimes a_n\\
& \qquad \qquad + (-1)^na_n\widetilde{a_0} \otimes a_1 \otimes \ldots \otimes a_{n-1}
\end{align*}
for $\widetilde{a_0} \in A_+$ and $a_i \in A$, and extended by continuity.  Then $b^2 = 0$, and the homology of the chain complex $(C_\bullet(A), b)$ is the Hochschild homology, denoted $HH_\bullet(A)$ or $HH_\bullet^R(A)$, for emphasis on the ground ring $R$.

Connes' differential $B: C_n(A) \to C_{n+1}(A)$ is given by
\[ B(a_0 \otimes \ldots \otimes a_n) = \sum_{j=0}^n (-1)^{jn} e \otimes a_j \otimes a_{j+1} \otimes \ldots \otimes a_n \otimes a_0 \otimes a_1 \otimes \ldots \otimes a_{j-1}, \]
\[ B(e \otimes a_1 \otimes \ldots \otimes a_n) = 0. \]
It satisfies $B^2 = 0$ and $bB + Bb = 0$.
The even and odd periodic cyclic chain groups of $A$ are
\[ C_{\even}(A) = \prod_{n=0}^\infty C_{2n}(A), \qquad C_{\odd}(A) = \prod_{n=0}^\infty C_{2n+1}(A). \]  They form a $Z/2$-graded chain complex
\[ \xymatrix{
C_{\even}(A) \ar@<.5ex>[r]^-{b+B} &C_{\odd}(A) \ar@<.5ex>[l]^-{b+B}
} \]
whose even and odd homology groups are the even and odd periodic cyclic homology groups, denoted $HP_0(A)$ and $HP_1(A)$ respectively.

The Hochschild cohomology $HH^\bullet(A)$ is the cohomology of the dual complex $C^\bullet(A) := (C_\bullet(A))^\dual$ with boundary map $b^\dual$, though we'll denote the boundary map simply as $b$.  Similarly, the periodic cyclic cohomology $HP^\bullet(A)$ is the cohomology of the dual complex of $C_{\per}(A)$.

We'll need another variant of Hochschild cohomology.  Let
\[ C^k(A,A) = \Hom_R(A^{\potimes_R k}, A) \]
and define $\delta: C^k(A,A) \to C^{k+1}(A,A)$ by
\begin{align*}
\delta D(a_1, \ldots, a_{k+1}) &= D(a_1, \ldots , a_k)a_{k+1} + (-1)^{k+1} a_1D(a_2, \ldots, a_{k+1})\\ & \qquad + \sum_{j=1}^k (-1)^{k-j+1}D(a_1, \ldots , a_{j-1}, a_ja_{j+1}, a_{j+2}, \ldots , a_k).
\end{align*}  Then $\delta^2 = 0$, and the cohomology of $C^\bullet(A,A)$ is the Hochschild cohomology of $A$ (with coefficients in $A$), denoted $H^\bullet(A,A)$.  Notice that $D \in C^1(A,A)$ satisfies $\delta D = 0$ if and only if $D$ is a derivation.

The complex $C^\bullet(A,A)$ and its cohomology has much additional structure \cite{MR0161898}.  There is an associative cup product
\[ \smile: C^k(A,A) \otimes C^\ell (A,A) \to C^{k+\ell }(A,A) \]
given by
\[ (D \smile E)(a_1, \ldots , a_{k+l}) = (-1)^{k\ell}D(a_1, \ldots , a_k)E(a_{k+1}, \ldots , a_{k+\ell}), \] which satisfies
\[ \delta(D \smile E) = (\delta D)\smile E ~+~ (-1)^k D \smile (\delta E). \]
There is also a Lie bracket
\[ [\cdot, \cdot]: C^k(A,A) \otimes C^\ell(A,A) \to C^{k+\ell-1}(A,A) \]
called the Gerstenhaber bracket.  It is given by
\[ [D, E] = D \circ E - (-1)^{(k-1)(\ell-1)} E \circ D, \] where
\[ (D \circ E)(a_1, \ldots , a_{k+l-1}) = \sum_{j=0}^{k-1} (-1)^{j(\ell-1)} D(a_1, \ldots , E(a_{j+1}, \ldots , a_{j+\ell}), \ldots , a_{k+\ell-1}). \]  When $D,E \in C^1(A,A)$, this bracket is just the usual commutator of linear maps.  If $m: A \potimes_R A \to A$ denotes the multiplication map, then
\[ \delta D = [m, D]. \]  The bracket also satsifies
\[ \delta [D,E] = [\delta D, E] + (-1)^{k-1}[D, \delta E]. \]  If we shift the degree by letting $\g^\bullet(A) = C^{\bullet+1}(A,A)$, then $\g^\bullet(A)$ is a differential graded Lie algebra.  With all of the structure above, the cohomology $H^\bullet(A,A)$ is a Gerstenhaber algebra.

\subsection{Chern-Connes character}

For more details on the Chern-Connes character, see \cite[Ch. 8]{MR1600246}.  The Chern-Connes character is a group homomorphism
\[ \ch: K_i(A) \to HP_i^R(A),\qquad i=0,1, \]
where $K_i(A)$ is the algebraic $K$-theory group of $A$.  We will need a description of this map.  Given an idempotent $P \in A$, the cycle $\ch P \in C_{\even}(A)$ is given by
\[ (\ch P)_{2n} = (-1)^n\frac{(2n)!}{n!}(P^{\otimes (2n+1)} - \frac{1}{2} e \otimes P^{\otimes (2n)}). \]  For an invertible $U \in A$, the cycle $\ch U \in C_{\odd}(A)$ is given by
\[ (\ch U)_{2n+1} = (-1)^n n! U^{-1} \otimes U \otimes U^{-1} \otimes \ldots \otimes U^{-1} \otimes U.\]  More generally if $P$ is an idempotent in a matrix algebra $M_N(A)$, then $\ch P \in C_{\even}(A)$ is obtained as the image of $\ch P \in C_{\even}(M_N(A))$ under the generalized trace
\[ T: C_{\per}(M_N(A)) \to C_{\per}(A), \]
see \cite[Ch. 1]{MR1600246}, and similarly for an invertible $U \in M_N(A)$.  The map $T$ is a chain homotopy equivalence which induces isomorphisms $HH_\bullet(M_N(A)) \cong HH_\bullet(A)$ and $HP_\bullet(M_N(A)) \cong HP_\bullet(A)$.

From the natural pairing between cyclic homology and cyclic cohomology, we obtain the Chern-Connes pairing
\[ \langle \cdot, \cdot \rangle: HP_R^i(A) \times K_i(A) \to R \] given by
\[ \langle [\phi], [P] \rangle = \langle [\phi], [\ch P] \rangle, \qquad \langle [\psi], [U] \rangle = \langle [\psi], [\ch U] \rangle. \]

\subsection{Operations on the cyclic complex} \label{Section-Operations}

Elements of $C^\bullet(A,A)$ act on $C_\bullet(A)$, and consequently $C_{\per}(A)$, by Lie derivative and contraction operations \cite{MR1261901}, \cite{MR2308582}.  We quickly review these operations and their properties.  Our conventions vary from \cite{MR1261901} and are closer to \cite{MR2308582}.  Proofs of all identities in this section, with these conventions, are given in \cite{Yashinski-Thesis}.

Given $D \in C^k(A,A)$, the \emph{Lie derivative} is the operator $L_D: C_n(A) \to C_{n-k+1}(A)$ of degree $-(k-1)$ given by
\begin{align*}
&L_D(a_0 \otimes \ldots \otimes a_n)\\
&\qquad = \sum_{i=0}^{n-k+1}(-1)^{i(k-1)}a_0 \otimes \ldots \otimes D(a_i, \ldots , a_{i+k-1}) \otimes \ldots \otimes a_n\\
& \qquad \qquad + \sum_{i=1}^{k-1}(-1)^{in}D(a_{n-i+1}, \ldots , a_n, a_0, \ldots , a_{k-1-i}) \otimes a_{k-i} \otimes \ldots \otimes a_{n-i}.
\end{align*}
In this formula and others below, it is understood that the result is $0$ whenever $D$ has the adjoined unit $e$ as one of its arguments.

\begin{proposition} \label{Proposition-LieDerivativeProperties} For any $D, E \in C^\bullet(A,A)$,
\[ [L_D, L_E] = L_{[D,E]},\qquad [b, L_D] = L_{\delta D}, \qquad [B, L_D] = 0. \]
\end{proposition}

Commutators of operators on the cyclic complex will always mean graded commutators.  So if $D \in C^k(A,A)$ and $E \in C^\ell(A,A)$,
\[ [L_D, L_E] = L_DL_E - (-1)^{(k-1)(\ell - 1)}L_E L_D, \]
and \[ [b+B, L_D] = (b+B)L_D - (-1)^{k-1}L_D(b+B). \]
The result of the proposition is that the periodic cyclic complex is a differential graded module over the differential graded Lie algebra $\g^\bullet(A)$.  A simple, but relevant, example is that if $D \in C^1(A,A)$, then
\[ L_D(a_0 \otimes \ldots \otimes a_n) = \sum_{j=0}^n a_0 \otimes \ldots \otimes a_{j-1} \otimes D(a_j) \otimes a_{j+1} \otimes \ldots \otimes a_n. \]

Given $D \in C^k(A,A)$, define operators
\[ \iota_D: C_n(A) \to C_{n-k}(A),\qquad S_D: C_n(A) \to C_{n-k+2}(A) \] by
\[ \iota_D(a_0 \otimes \ldots \otimes a_n) = (-1)^{k-1}a_0D(a_1, \ldots , a_k) \otimes a_{k+1} \otimes \ldots \otimes a_n, \]
\begin{align*}
&S_D(a_0 \otimes \ldots \otimes a_n) = \sum_{i=1}^{n-k+1}\sum_{j=0}^{n-i-k+1} (-1)^{i(k-1) + j(n-k+1)}\\
& \qquad e \otimes a_{n-j+1} \otimes \ldots \otimes a_n \otimes a_0 \otimes \ldots \otimes a_{i-1} \otimes D(a_i, \ldots , a_{i+k-1}) \otimes a_{i+k} \otimes \ldots \otimes a_{n-j},
\end{align*}
\[ S_D(e \otimes a_1 \otimes \ldots \otimes a_n) = 0. \]
Define the \emph{contraction} on $C_{\per}(A)$ to be the operator
\[ I_D = \iota_D + S_D. \]
Now we can state Getzler's Cartan homotopy formula \cite{MR1261901}.

\begin{theorem}[Cartan homotopy formula] \label{Theorem-CHF} For any $D \in C^\bullet(A,A)$,
\[ [b+B, I_D] = L_D - I_{\delta D}. \]
\end{theorem}

In particular, if $D \in C^\bullet(A,A)$ is a Hochschild cocycle, then $L_D = 0$ on $HP_\bullet(A)$.

The Lie derivative and contraction operations of the previous section have multiple generalizations, see e.g. \cite{MR1261901} or \cite{MR2308582}.  We shall need just one of these.  For $X,Y \in C^1(A,A)$, define the operators $L\{X,Y\}$ and $I\{X,Y\}$ on $C_{\bullet}(A)$ by
\begin{align*}
L\{X,Y\}(a_0 \otimes \ldots \otimes a_n) &= \sum_{i=1}^{n-1} \sum_{j=i+1}^n a_0 \otimes \ldots \otimes X(a_i) \otimes \ldots \otimes Y(a_j) \otimes \ldots \otimes a_n)\\
& \qquad + \sum_{i=1}^n Y(a_0) \otimes a_1 \otimes \ldots \otimes X(a_i) \otimes \ldots \otimes a_n.
\end{align*}
and
\begin{align*} &I\{X,Y\}(a_0 \otimes \ldots \otimes a_n)\\ &\qquad = \sum_{i=1}^{n-1} \sum_{j=i+1}^n \sum_{m=0}^{n-j} (-1)^{nm}\\  &\qquad \qquad e \otimes a_{n-m+1} \otimes \ldots \otimes a_n \otimes a_0 \otimes \ldots \otimes X(a_i) \otimes \ldots \otimes Y(a_j) \otimes \ldots \otimes a_{n-m}
\end{align*}
if $a_0 \in A$ and
\[ I\{X,Y\}(e \otimes a_1 \otimes \ldots \otimes a_n) = 0. \]
The following formulas appear in \cite{MR1261901}, with slightly different conventions.

\begin{theorem} \label{Theorem-CHF-TwoDerivations}
If $X$ and $Y$ are derivations, then
\begin{enumerate}[(i)]
\item $[b+B, I\{X,Y\}] = L\{X,Y\} + I_{X \smile Y} - I_Y I_X$.
\item $[b+B, L\{X,Y\} ] = -L_{X \smile Y} + L_Y I_X - I_Y L_X$.
\end{enumerate}
\end{theorem}
Notice that the second formula follows from the first by applying the commutator with $b+B$ and using Theorem~\ref{Theorem-CHF}.

\subsection{Smooth deformations}

We work in the setting of \cite{Yashinski-SmoothDefs}.  Let $J \subseteq \R$ be an interval.  By a \emph{smooth deformation of Fr\'{e}chet algebras}, we mean a Fr\'{e}chet space $X$ equipped with a collection $\{ m_t: X \times X \to X\}_{t \in J}$ of continuous associative multiplications which depend smoothly on $t$ in the sense that for every $x_1,x_2  \in X$, the map $t \mapsto m_t(x_1,x_2)$ is smooth, i.e. infinitely differentiable.  Given such a smooth deformation, let $A_t$ denote the Fr\'{e}chet algebra $(X, m_t)$ for each $t \in J$.  We will often refer to $\{A_t\}_{t \in J}$ as the smooth deformation.  Let $A_J = C^\infty(J,X)$ denote the space of all smooth functions from the interval $J$ into $X$.  This space is a Fr\'{e}chet space under the usual topology of uniform convergence of functions and all their derivatives on compact subsets of $J$.  If we view the vector spaces $\{A_t\}_{t \in J}$ as forming a vector bundle over $J$, then $A_J$ is the space of smooth sections of this bundle.  Consequently, $A_J$ is a topological $C^\infty(J)$-module, where the action is given by pointwise scalar multiplication.  We equip $A_J$ with the multiplication
\[ (a_1 a_2)(t) = m_t(a_1(t), a_2(t)),\qquad a_1, a_2 \in A_J. \]
As shown in \cite{Yashinski-SmoothDefs}, $A_J$ is a Fr\'{e}chet $C^\infty(J)$-algebra.  Conversely, any continuous $C^\infty(J)$-linear multiplication on $C^\infty(J,X)$ arises from a smooth deformation.

\subsection{Connections and parallel translation}
By a \emph{connection} on a locally convex $C^\infty(J)$-module $M$, we mean a continuous $\C$-linear map $\nabla: M \to M$ such that
\[ \nabla(f\cdot m) = f' \cdot m + f\cdot \nabla(m), \qquad f \in C^\infty(J), ~m \in M \]
Notice that $\frac{d}{dt}$ is a connection on any module of the form $C^\infty(J,X)$ for some locally convex vector space $X$.  Any other connection on $C^\infty(J,X)$ is of the form $\nabla = \frac{d}{dt} - F$ for some continuous $C^\infty(J)$-linear endomorphism $F$ of $C^\infty(J,X)$.

Given two $C^\infty(J)$-modules $M$ and $N$ with connections $\nabla_M$ and $\nabla_N$, a $C^\infty(J)$-linear map $F: M \to N$ is called \emph{parallel} if $F \circ \nabla_M = \nabla_N \circ F$.  A connection $\nabla_M$ on $M$ is called \emph{integrable} if there is a parallel isomorphism
\[ F: (M, \nabla_M) \to \left(C^\infty(J,X), \frac{d}{dt}\right) \]
of locally convex $C^\infty(J)$-modules, for some $X \in \LCTVS$.

Integrability of a connection can be characterized in terms of parallel translation.  Let $M = C^\infty(J,X)$ and let $\nabla$ be a connection on $M$.  For each $t \in J$, let $M_t = X$, which we think of as the fiber of $t \in J$.  Suppose that for every $s \in J$ and $x \in M_s$, there exists a unique $m \in M$ such that
\[ \nabla m = 0,\qquad m(s) = x. \]
Then we can define a parallel translation map
\[ P^\nabla_{s,t}: M_s \to M_t, \qquad P^\nabla_{s,t}(x) = m(t), \]
where $m \in M$ is the unique solution as above.  Then $P^\nabla_{s,t}$ is a linear isomorphism with inverse $P^\nabla_{t,s}$.

\begin{theorem} [\cite{Yashinski-SmoothDefs}] \label{Theorem-IntegrabilityCriteria}
If $X$ is a Fr\'{e}chet space, then a connection $\nabla$ on $M = C^\infty(J,X)$ is integrable if and only if the following two conditions hold:
\begin{enumerate}[(i)]
\item For every $s \in J$ and $x \in M_s$, there is a unique $m \in M$ such that
\[ \nabla m = 0, \qquad m(s) = x. \]
\item Each $P^\nabla_{s,t}: M_s \to M_t$ is continuous, and for each fixed $x \in X$, the map $(s,t) \mapsto P^\nabla_{s,t}(x)$ is smooth (i.e. all mixed partial derivatives exist).
\end{enumerate}
\end{theorem}

In this case, each $P_{s,t}^\nabla$ is continuous, and therefore is an isomorphism of topological vector spaces.

Given a connection $\nabla$ on $M$, the dual connection $\nabla^\dual$ on $M^\dual = \Hom_{C^\infty(J)}(M, C^\infty(J))$ is
\[ (\nabla^\dual \phi)(m) = \frac{d}{dt}(\phi(m)) - \phi(\nabla m), \qquad \phi \in M^\dual, m \in M. \]  By design, the canonical pairing
\[ \langle \cdot , \cdot \rangle: M^\dual \otimes M \to C^\infty(J) \] is parallel, meaning
\[ \frac{d}{dt} \langle \phi, m \rangle = \langle \nabla^\dual \phi, m \rangle + \langle \phi, \nabla m \rangle. \]  As shown in \cite{Yashinski-SmoothDefs}, if $M = C^\infty(J,X)$ for a nuclear Fr\'{e}chet space $X$, then $M^\dual \cong C^\infty(J,X^*)$.  Moreover if $\nabla$ is an integrable connection on $M$, then the dual connection $\nabla^\dual$ is integrable on $M^\dual$, and
\[ P^{\nabla^\dual}_{s,t} = (P^\nabla_{t,s})^*: M_s^* \to M_t^*. \]

\subsection{The Gauss-Manin connection} \label{Section-GaussManinConnection}

The Gauss-Manin connection was initially introduced by Getzler in the setting of formal deformations \cite{MR1261901}.  We shall use it in the smooth setting, as in \cite{Yashinski-SmoothDefs}.

Let $\{A_t\}_{t \in J}$ be a smooth deformation of Fr\'{e}chet algebras and let $A_J$ denote the algebra of sections.  Consider the complex $C_{\per}^{C^\infty(J)}(A_J)$, which is naturally isomorphic as a $C^\infty(J)$-module to $C^\infty(J, C_{\per}(X))$, where $X$ is the underlying Fr\'{e}chet space of the algebras in the deformation.  So the complex $C_{\per}^{C^\infty(J)}(A_J)$ can be viewed as the space of sections of a bundle of chain complexes, and the Gauss-Manin connection a certain connection on this bundle.  To construct it, choose any connection $\nabla$ on $A_J$ and let $E = \delta \nabla \in C^2(A_J,A_J)$, so that
\[ \nabla(a_1a_2) = \nabla(a_1)a_2 + a_1\nabla(a_2) - E(a_1, a_2),\qquad a_1, a_2 \in A_J. \]  The Gauss-Manin connection is defined as
\[ \nabla_{GM} = L_\nabla - I_E: C_{\per}^{C^\infty(J)}(A_J) \to C_{\per}^{C^\infty(J)}(A_J). \]
Since $\nabla$ is a connection, it follows that $E$ is $C^\infty(J)$-bilinear and the operator $I_E$ is well-defined on $C_{\per}^{C^\infty(J)}(A_J)$ as in section \ref{Section-Operations}.  The connection $\nabla$ is not $C^\infty(J)$-linear, but the formula for $L_\nabla$ still gives a well-defined operator on the quotient complex $C_{\per}^{C^\infty(J)}(A_J)$ thanks to the Leibniz rule.  In fact, $L_\nabla$ itself is a connection.

The connection $\nabla_{GM}$ commutes with the boundary map $b+B$:
\[ [b+B, \nabla_{GM}] = [b+B, L_\nabla] - [b+B, I_E] = L_{\delta \nabla} - (L_E - I_{\delta E}) = 0 \]
because $\delta E = \delta^2 \nabla = 0$.
Consequently, $\nabla_{GM}$ induces a connection on $HP_\bullet^{C^\infty(J)}(A_J)$.  As shown in \cite{Yashinski-SmoothDefs}, the induced connection is independent of the initial choice of connection $\nabla$ on $A_J$.

If $\nabla_{GM}$ is integrable as a connection on $C_{\per}^{C^\infty(J)}(A_J)$, then its parallel translation maps
\[ P_{s,t}^{\nabla_{GM}}: C_{\per}(A_s) \to C_{\per}(A_t) \]
are automatically chain maps, hence isomorphisms of chain complexes.  Integrability of $\nabla_{GM}$ at the level of chain complexes seems unlikely in practice.  Alternatively, we can ask if $\nabla_{GM}$ is integrable as a connection on the homology module $HP_\bullet^{C^\infty(J)}(A_J)$.  This module may be quite pathological as both a topological vector space and a $C^\infty(J)$-module.  We can still try to solve the parallel translation equation
\[ \nabla_{GM}[c] = [0], \qquad [c(s)] = [c_s] \]
for $[c] \in HP_\bullet^{C^\infty(J)}(A_J)$, given $s \in J$ and initial data $[c_s] \in HP_\bullet(A_s)$.  If solutions exist and are unique, then there are parallel translation isomorphisms
\[ P^{\nabla_{GM}}_{s,t}: HP_\bullet(A_s) \to HP_\bullet(A_t). \]

\subsection{Deformation of Chern-Connes pairing}

We define the dual Gauss-Manin connection to be the dual connection $\nabla^{GM} = (\nabla_{GM})^\dual$ on $C^{\per}_{C^\infty(J)}(A_J)$, as defined earlier.  It follows that
\[ \frac{d}{dt} \langle [\phi], [\omega] \rangle = \langle \nabla^{GM} [\phi], [\omega] \rangle + \langle [\phi], \nabla_{GM} [\omega] \rangle, \qquad [\phi] \in HP^\bullet(A_J), [\omega] \in HP_\bullet(A_J). \]
It was shown in \cite{Yashinski-SmoothDefs} that any $[\omega] \in HP_\bullet^{C^\infty(J)}(A_J)$ in the image of the Chern-Connes character satisfies $\nabla_{GM}[\omega] = 0$ in $HP_\bullet^{C^\infty(J)}(A_J)$.  For the Chern-Connes pairing, we obtain
\[ \frac{d}{dt} \langle [\phi], [P] \rangle = \langle \nabla^{GM}[\phi], [P] \rangle, \qquad \frac{d}{dt} \langle [\psi], [U] \rangle = \langle \nabla^{GM}[\psi], [U] \rangle. \]

\section{Smooth noncommutative tori}
Given an $n \times n$ skew-symmetric real-valued matrix $\Theta$, the \emph{noncommutative torus} $A_\Theta$ is the universal $C^*$-algebra generated by $n$ unitaries $u_1, \ldots , u_n$ such that
\[ u_j u_k = e^{2\pi i\theta_{jk}}u_k u_j, \] where $\Theta = (\theta_{jk})$, see \cite{MR1047281}.  In the case $\Theta = 0$, all of the generating unitaries commute, and we have $A_0 \cong C(\T^n)$, the algebra of continuous complex-valued functions on the $n$-torus.  It is for this reason why the algebra $A_\Theta$ has earned its name, as it can be philosophically viewed as functions on some ``noncommutative torus" in the spirit of Alain Connes' noncommutative geometry \cite{MR1303779}.

It is helpful to think of an element in $x \in A_\Theta$ formally as a ``Fourier series"
\[ x = \sum_{\alpha \in \Z^n} x_\alpha u^\alpha, \]
where $\alpha = (\alpha_1, \ldots \alpha_n)$ is a multi-index, $u^\alpha = u_1^{\alpha_1}u_2^{\alpha_2}\ldots u_n^{\alpha_n}$, and $x_\alpha \in \C$.  Indeed, the isomorphism $C(\T^n) \cong A_0$ is essentially implemented by mapping a function to its Fourier series.  There are many issues with the convergence of Fourier series of continuous functions, so the above series will typically not converge in norm.  This will not be an issue for us because we will work exclusively with smooth functions, and their noncommutative analogues, which have well-behaved Fourier series.

As a topological vector space, the \emph{smooth noncommutative torus} $\A_\Theta$ is the Schwarz space $\S(\Z^n)$ of complex-valued sequences indexed by $\Z^n$ of rapid decay, defined as follows.  Given a multi-index $\alpha = (\alpha_1, \ldots , \alpha_n) \in \Z^n$, we shall write
\[ |\alpha| = |\alpha_1| + \ldots + |\alpha_n|. \]
A sequence $x = (x_\alpha)_{\alpha \in \Z^n}$ is of \emph{rapid decay} if for every positive integer $k$,
\[ p_k(x) := \sum_{\alpha \in \Z^n}(1 + |\alpha|)^k|x_\alpha| < \infty. \]
The functions $p_k$ are seminorms, and we give $\A_\Theta$ the locally convex topology defined by these seminorms.  Under this topology, $\A_\Theta$ is complete, and therefore is a Fr\'{e}chet space.
We identify $\A_\Theta$ as a subspace of $A_\Theta$ by the map $\iota: \A_\Theta \to A_\Theta$,
\[ \iota(x) = \sum_{\alpha \in \Z^n}x_\alpha u^\alpha. \]
Then $\A_\Theta$ is norm dense in $A_\Theta$ because it contains the $*$-algebra generated by $u_1, \ldots , u_n$.  The multiplication in $\A_\Theta$ is given by the twisted convolution product
\[ (xy)_\alpha = \sum_{\beta \in \Z^n}e^{2\pi iB_\Theta(\alpha - \beta, \beta)}x_{\alpha - \beta}y_\beta, \] where
\[ B_\Theta(\alpha, \beta) = \sum_{j > k} \alpha_j \beta_k \theta_{jk} \] is the associated $\R$-valued group $2$-cocycle on $\Z^n$.
One can show that for any $k$, \[ p_k(xy) \leq p_k(x)p_k(y), \qquad x,y \in \A_\Theta. \]  This shows that $\A_\Theta$ is a subalgebra, and in fact is an $m$-convex Fr\'{e}chet algebra.

The algebra $\A_\Theta$ possesses $n$ canonical continuous derivations 
\[\delta_1,\ldots , \delta_n: \A_\Theta \to \A_\Theta\] defined by
\[ (\delta_j(x))_\alpha = 2\pi i\alpha_j \cdot x_\alpha. \]
In the case $\theta = 0$, the derivations $\delta_1, \ldots , \delta_n$ correspond to the usual $n$ partial differentiation operators on $\T^n$.
There is also a canonical continuous trace $\tau: \A_\Theta \to \C$ given by $\tau(x) = x_0$, which corresponds to integration with respect to the normalized Haar measure in the case $\Theta = 0$.  Notice that $\tau \circ \delta_j = 0$ for all $j$.

The smooth noncommutative torus $\A_\Theta$ can be viewed as a smooth one-parameter deformation of $C^\infty(\T^n) \cong \A_0$ in the following way.  For each $t \in J = \R$, let $A_t = \A_{t\Theta}$.  The product in $A_t$ is given by
\[ m_t(x,y)_\alpha = \sum_{\beta \in \Z^n}e^{2\pi iB_\Theta(\alpha - \beta, \beta)t}x_{\alpha - \beta}y_\beta. \]

\begin{proposition} \label{Proposition-NCToriIsSmoothDeformation}
Given an $n \times n$ skew-symmetric real matrix $\Theta$, the deformation $\{\mathcal{A}_{t\Theta}\}_{t \in \R}$ is smooth, and for $x,y$ in the underlying space $\S(\Z^n)$,
\[ \frac{d}{dt} m_t(x,y) = \frac{1}{2\pi i} \sum_{j > k} \theta_{jk}\cdot m_t(\delta_j(x), \delta_k(y)). \]
\end{proposition}

\begin{proof}
We first derive an estimate that we will need.
Taylor's formula for a function $f$ of a real variable is
\[ f(t) = f(0) + f'(0)t + \int_0^t f''(u)(t-u)du. \]
Applying this to $f(t) = e^{2\pi iat}$ for a fixed $a\in \R$, we obtain the estimate
\[ \left| \frac{e^{2\pi iat} - 1}{t} - 2\pi ia \right| \leq 2\pi^2a^2t.\]  The importance is that the bound is polynomial in $a$.  This would fail if $a \notin \R$.

To prove the deformation is smooth, we must prove $t \mapsto m_t(x,y)$ is smooth for all $x,y \in \S(\Z^n)$.  For fixed $x,y$, and $t$, let
\[ z_h = \frac{m_{t+h}(x,y) - m_t(x,y)}{h} - \frac{1}{2\pi i} \sum_{j > k} \theta_{jk} \cdot m_t(\delta_j(x), \delta_k(y)),\] and we shall show $z_h \to 0$ in $\S(\Z^n)$.  Then
\[ (z_h)_\alpha = \sum_{\beta \in \Z^n} e^{2\pi i B_\Theta(\alpha - \beta, \beta)t}\left( \frac{e^{2\pi i B_\Theta(\alpha - \beta, \beta)h} - 1}{h} - 2\pi i B_\Theta(\alpha - \beta, \beta) \right)x_{\alpha - \beta}y_\beta. \]
Using the above estimate, for a nonnegative integer $r$ we obtain
\begin{align*}
p_r(z_h) &= \sum_{\alpha \in \Z^n}(1 + |\alpha|)^r|(z_h)_\alpha|\\
&\leq \sum_{\alpha \in \Z^n} \sum_{\beta \in \Z^n}(1 + |\alpha|)^r 2\pi^2 h B_\Theta(\alpha-\beta, \beta)^2|x_{\alpha-\beta}||y_\beta|.
\end{align*}
Consider the map $D_\Theta: \S(\Z^n) \otimes \S(\Z^n) \to \S(\Z^n) \otimes \S(\Z^n)$ given by
\[ x \otimes y \mapsto \frac{1}{2\pi i} \sum_{j>k} \theta_{jk} \cdot \delta_j(x) \otimes \delta_k(y). \]  Notice that
\[ m_t(D_\Theta^2(x\otimes y))_\alpha = (2\pi i)^2 \sum_{\beta \in \Z^n} e^{2\pi i B_\Theta(\alpha - \beta, \beta)} B_\Theta(\alpha - \beta, \beta)^2 x_{\alpha - \beta}y_\beta. \]  So the last inequality can be rewritten as
\[ p_r(z_h) \leq \frac{h}{2} p_r(m_t(D_\Theta^2(x \otimes y))). \]  Thus, $p_r(z_h) \to 0$ as $h \to 0$ because $p_r(m_t(D_\Theta^2(x \otimes y)))$ is a finite quantity, independent of $h$.  This proves the formula
\[ \frac{d}{dt} m_t(x,y) = \frac{1}{2\pi i} \sum_{j > k} \theta_{jk} \cdot m_t(\delta_j(x), \delta_k(y)). \]  By induction, it follows from this formula that $t \mapsto m_t(x,y)$ is infinitely differentiable.
\end{proof}

Let $A_J$ be the algebra of sections of the smooth deformation $\{\A_{t\Theta}\}_{t \in J}$.  Notice that the derivations $\delta_1, \ldots , \delta_n$ on the fibers $\A_{t\Theta}$ induce $C^\infty(J)$-linear derivations on $A_J$, which we also denote $\delta_1, \ldots , \delta_n$.  Let $\nabla = \frac{d}{dt}$ be the canonical connection on $A_J$.  We conclude that the cocycle $E = \delta \nabla$, as in section \ref{Section-GaussManinConnection}, is given by
\[ E = \frac{1}{2\pi i} \sum_{j > k} \theta_{jk} \cdot \delta_j \smile \delta_k. \]
Indeed, both sides are $C^\infty(J)$-bilinear and continuous, and they agree on pairs of constant elements $(x,y)$ by Proposition~\ref{Proposition-NCToriIsSmoothDeformation}.  Thus they are equal because the $C^\infty(J)$-linear span of the constants is dense in $A_J \potimes_{C^\infty(J)} A_J$.  So we have the identity
\[ \nabla(a_1a_2) = \nabla(a_1)a_2 + a_1\nabla(a_2) + \frac{1}{2\pi i}\sum_{j>k}\theta_{jk}\delta_j(a_1)\delta_k(a_2), \qquad a_1, a_2 \in A_J. \]
One can show $[E] \neq 0$ in $H^2_{C^\infty(J)}(A_J,A_J)$, from which it follows that this deformation is nontrivial.  In fact, $[E_t] \neq 0 \in H^2(\A_{t\Theta}, \A_{t\Theta})$, which implies the deformation is locally nontrivial at each fiber (though this is well-known).

\section{The $\g$-invariant cyclic complex} \label{Section-InvariantComplex} \label{Section-gInvariantComplex}

In this section, we consider a Fr\'{e}chet $R$-algebra $A$ with an action of an abelian Lie algebra $\g$ by derivations.  Our main example is the noncommutative $n$-torus $\A_\Theta$ and $\g = \Span\{\delta_1, \ldots , \delta_n\}$.  We can also consider the algebra of sections $A_J$ of a noncommutative torus deformation $\{\A_{t\Theta}\}_{t \in J}$, where $\g$ is the $\C$-linear span of $\{\delta_1, \ldots , \delta_n\}$, viewed as operators on $A_J$.  The results of this section do not depend on any additional structure of these examples.

\subsection{$\g$-invariant chains and cochains}
Suppose that $\g \subset C^1_R(A,A)$ is a complex abelian Lie subalgebra of $R$-linear derivations on a Fr\'{e}chet $R$-algebra $A$.  Then $\g$ also acts on $C_{\bullet}^R(A)$ by Lie derivatives.
Define the \emph{$\g$-invariant Hochschild chain group} $C_\bullet^\g(A)$ to be the space of coinvariants of this action, that is
\[C_\bullet^\g(A) = C_\bullet^R(A)/\g \cdot C_\bullet^R(A). \]  For simplicity, we suppress the ground ring $R$ from the notation.  We shall make the assumption that $\g \cdot C_\bullet^R(A)$ is closed submodule, as this holds in the examples which are of interest to us.  Thus, $C_\bullet^\g(A)$ is Hausdorff.
That $X \in \g$ is a derivation on $A$ implies that $\delta X = 0$.  So by Proposition~\ref{Proposition-LieDerivativeProperties}, the operators $b$ and $B$ descend to operators on $C_\bullet^\g(A)$.  One can define the $\g$-invariant periodic cyclic complex $C_{\per}^\g(A)$ as the product of $\g$-invariant Hochschild chain groups, and its homology is the \emph{$\g$-invariant periodic cyclic homology} $HP_\bullet^\g(A)$.

Let $C^\bullet_\g(A,A)$ denote the space of all Hochschild cochains $D$ for which $[X,D] = 0$ for all $X \in \g$.  Since $\g$ is abelian, we have $\g \subseteq C^1_\g(A,A)$.  If $D \in C^\bullet_\g(A,A)$, then the formula
\[ 0 = \delta[X,D] = [\delta X, D] + [X, \delta D] \] shows that $C^\bullet_\g(A,A)$ is a subcomplex because $\delta X = 0$.

\begin{proposition} \label{Proposition-OperationsRestrictToInvariantComplex}
For any $D \in C^\bullet_\g(A)$, the operators $L_D$ and $I_D$ are well-defined on the $\g$-invariant complex $C_\bullet^\g(A)$.
\end{proposition}

\begin{proof}
If $X \in C^1(A,A)$ is a derivation and $D \in C^\bullet(A,A)$, then one can verify directly that
\[ [L_X, I_D] = I_{[X,D]}, \qquad [L_X, L_D] = L_{[X,D]}, \]
from which the proposition follows.
\end{proof}

\begin{proposition}
If $X,Y \in C^1_{\g}(A,A)$, then $L\{X,Y\}$ and $I\{X,Y\}$ are well-defined operators on $C_\bullet^{\g}(A)$.
\end{proposition}

\begin{proof}
For any $X,Y,Z \in C^1(A,A)$, one can verify directly the identities
\[ [L_Z, I\{X,Y\}] = I\{[Z,X], Y\} + I\{X, [Z,Y]\} \] and
\[ [L_Z, L\{X,Y\}] = L\{[Z,X], Y\} + L\{X, [Z,Y]\}, \] and the result follows.
\end{proof}


One of the benefits of working in the $\g$-invariant complex is that the contraction $I_X$ is now a chain map when $X \in \g$.  Indeed,
\[ [b+B, I_X] = L_X = 0 \] in $C_\bullet^{\g}(A)$.  These contraction operators obey the following algebra as operators on homology.

\begin{theorem} \label{Theorem-CupProduct}
There is an algebra map $\chi: \Lambda^\bullet \g \to \End(HP_\bullet^\g(A))$ given by
\[ \chi(X_1 \wedge X_2 \wedge \ldots X_k) = I_{X_1} I_{X_2} \ldots I_{X_k}. \]
\end{theorem}

\begin{proof}
First notice that $X \mapsto I_X$ is a linear mapping.  Next, we shall show that $I_X I_X$ is chain homotopic to zero.  On $C_\bullet^\g(A)$, observe that \[ 0 = L_X L_X = L_{X^2} + 2L\{X, X\} \] where $X^2$ denotes the composition of $X$ with itself.  Thus, \[ L\{X,X\} = -L_{\frac{1}{2}X^2}. \]  Next, notice that \[ \delta\left(\frac{1}{2} X^2\right) = X\smile X. \] By Theorems~\ref{Theorem-CHF-TwoDerivations} and \ref{Theorem-CHF},
\begin{align*}
-[b+B, I\{X,X\}] &= -L\{X,X\} - I_{X \smile X} + I_X I_X\\
&= L_{\frac{1}{2}X^2} - I_{\delta(\frac{1}{2}X^2)} + I_X I_X\\
&= [b+B, I_{\frac{1}{2}X^2}] + I_X I_X,
\end{align*} which proves that $I_X I_X$ is continuously chain homotopic to zero.  By the universal property of the exterior algebra, the map $\chi$ exists as asserted.
\end{proof}

There are some additional simplifications regarding the operator $L\{X,Y\}$ once we pass to $C_\bullet^{\g}(A)$.

\begin{proposition} \label{Proposition-LieDerivative-TwoDerivationsInvariant}
For $X,Y \in \g$, the operator $L\{X,Y\}$ satisfies
\[ L\{X,Y\}(a_0 \otimes \ldots \otimes a_n) = \sum_{i=0}^{n-1} \sum_{j=i+1}^n a_0 \otimes \ldots \otimes X(a_i) \otimes \ldots \otimes Y(a_j) \otimes \ldots \otimes a_n \] on $C_{\bullet}^{\g}(A).$
Additionally,
\[ [b+B, L\{X,Y\} ] = -L_{X \smile Y} \] on $C_{\bullet}^{\g}(A).$
\end{proposition}

\begin{proof}
Notice that
\begin{align*}
&L_Y(X(a_0), a_1, \ldots , a_n) - L_X(Y(a_0), a_1, \ldots , a_n)\\
&\quad = \sum_{j=1}^nX(a_0) \otimes \ldots \otimes Y(a_j) \otimes \ldots \otimes a_n - \sum_{i=1}^n Y(a_0) \otimes \ldots \otimes X(a_i) \otimes \ldots \otimes a_n,\\
\end{align*} using the fact that $[X,Y] = 0$.  So
\begin{align*}
&L\{X,Y\}(a_0 \otimes \ldots \otimes a_n) + L_Y(X(a_0) \otimes a_1\otimes \ldots \otimes a_n) - L_X(Y(a_0) \otimes a_1 \otimes \ldots \otimes a_n)\\
& \qquad = \sum_{i=0}^{n-1} \sum_{j=i+1}^n a_0 \otimes \ldots \otimes X(a_i) \otimes \ldots \otimes Y(a_j) \otimes \ldots \otimes a_n
\end{align*} gives the desired conclusion.

The formula \[ [b+B, L\{X,Y\} ] = -L_{X \smile Y} \] follows from Theorem \ref{Theorem-CHF-TwoDerivations} in light of the fact that $L_X = L_Y = 0$ on $C_\bullet^{\g}(A)$.
\end{proof}

\subsection{Connections on the $\g$-invariant complex} \label{Subsection-ConnectionsOngInvariantComplex}

The following situation is modeled on what we see with the noncommutative tori deformation.  Suppose $A_J$ is the algebra of sections of a smooth deformation of Fr\'{e}chet algebras $\{A_t\}_{t \in J}$.  Suppose $\g$ is an abelian Lie algebra of $C^\infty(J)$-linear derivations on $A_J$.  Suppose $\nabla$ is a $\g$-invariant connection on $A_J$ for which
\[ E := \delta \nabla = \sum_{i=1}^r X_i \smile Y_i, \] where $X_i, Y_i \in \g$.  This means that
\[ \nabla(a_1 a_2) = \nabla(a_1)a_2 + a_1\nabla(a_2) + \sum_{i=1}^r X_i(a_1)Y_i(a_2),\qquad a_1,a_2 \in A_J. \]


\begin{proposition}
In the above situation, the Gauss-Manin connection
\[ \nabla_{GM} = L_\nabla - I_E \] descends to the $\g$-invariant complex $C_\bullet^\g(A_J)$ and therefore to a connection on the $\g$-invariant periodic cyclic homology $HP_\bullet^\g(A_J)$.
\end{proposition}

\begin{proof}
Since $\nabla$ is $\g$-invariant, so is $E = \delta \nabla$ because $\g$-invariant cochains form a subcomplex.  The statement follows from Proposition \ref{Proposition-OperationsRestrictToInvariantComplex}.
\end{proof}

Our main reason for working with the $\g$-invariant complex is that we can define another connection on $HP_\bullet^{\g}(A_J)$ which is easier to work with than $\nabla_{GM}$.  Recall that the connection $L_{\nabla}$ satisfies
\[ [b+B, L_\nabla] = L_E = \sum_{i=1}^r L_{X_i \smile Y_i}. \]
Thus, by Proposition~\ref{Proposition-LieDerivative-TwoDerivationsInvariant},
\[ \widetilde{\nabla} = L_\nabla + \sum_{i=1}^r L\{X_i,Y_i\} \]
is a connection on $C_{\per}^\g(A_J)$ that commutes with $b+B$ and therefore descends to a connection on $HP_\bullet^{\g}(A_J)$.  We emphasize that $\widetilde{\nabla}$ does not commute with $b+B$ on the ordinary periodic cyclic complex $C_{\per}(A_J)$.

\begin{remark}
There is a conceptual explanation for the form of this connection $\widetilde{\nabla}$.  Let $\H$ be the Hopf algebra whose underlying algebra is the symmetric algebra $S(\g \oplus \Span\{\nabla\})$.  The coproduct $\Delta: \H \to \H \otimes \H$ is the unique algebra map that satisfies
\[ \Delta(X_i) = X_i \otimes 1 + 1 \otimes X_i, \qquad \Delta(Y_i) = Y_i \otimes 1 + 1 \otimes Y_i, \]
\[ \Delta(\nabla) = \nabla \otimes 1 + 1 \otimes \nabla + \sum_{i=1}^r X_i \otimes Y_i. \]
In the $r = 1$ case, these are the defining relations of the Hopf algebra of polynomial functions on the three-dimensional Heisenberg group.  As an algebra, $\H$ acts on $A_J$ in the obvious way, and this action is a Hopf action in the sense that for all $h \in \H$,
\[ h(a_1a_2) = \sum h_{(1)}(a_1)h_{(2)}(a_2), \] where $\Delta(h) = \sum h_{(1)} \otimes h_{(2)}.$  When a Hopf algebra $\H$ acts (say, on the left) on two spaces $V$ and $W$, there is a canonical action, called the \emph{diagonal action}, of $\H$ on $V \otimes W$ given by
\[ h(v \otimes w) = \sum h_{(1)}(v) \otimes h_{(2)}(w). \]  The connection $\widetilde{\nabla}$ on $C_n^{\g}(A_J)$ is none other than the diagonal action of $\nabla$ on $A_J^{\potimes(n+1)}$ after passing to the quotient.
\end{remark}

\begin{lemma} \label{Lemma-ConnectionsCommute}
On the invariant complex $C_\bullet^\g(A_J)$, $[\nabla_{GM}, \widetilde{\nabla}] = 0.$
\end{lemma}

\begin{proof}
This is a straightforward, though tedious, computation.  For details, see \cite{Yashinski-Thesis}.
\end{proof}

\begin{proposition} \label{Proposition-NablaGMIsNilpotentPerturbation}
As operators on $HP_\bullet^{\g}(A_J)$,
\[ \nabla_{GM} = \widetilde{\nabla} + \sum_{i=1}^r \chi(X_i \wedge Y_i). \]
\end{proposition}

\begin{proof}
We have
\begin{align*}
\nabla_{GM} - \widetilde{\nabla} &= -I_E - \sum_{i=1}^r L\{X_i,Y_i\}\\
&= -\sum_{i=1}^r \Big( I_{X_i \smile Y_i} + L\{X_i,Y_i\} \Big)\\
&= -\sum_{i=1}^r \Big([b+B, I\{X_i,Y_i\}] + I_{Y_i}I_{X_i}\Big)\\
&= -\sum_{i=1}^r [b+B, I\{X_i,Y_i\}] - \sum_{i=1}^r \chi(Y_i \wedge X_i),
\end{align*} using Theorem~\ref{Theorem-CHF-TwoDerivations}.  So at the level of homology,
\[ \nabla_{GM} = \widetilde{\nabla} + \sum_{i=1}^r \chi(X_i \wedge Y_i). \]
\end{proof}

Let us use the notation $\Omega = \sum_{i=1}^r X_i \wedge Y_i \in \Lambda^2 \g$.  So in $HP_\bullet^{\g}(A_J)$,
\[ \nabla_{GM} = \widetilde{\nabla} + \chi(\Omega). \]

\begin{theorem} \label{GMiffTildeNabla}
As connections on $HP_\bullet^{\g}(A_J)$, $\nabla_{GM}$ is integrable if and only if $\widetilde{\nabla}$ is integrable.
\end{theorem}

\begin{proof}
Since $\Omega$ is a nilpotent element of the algebra $\Lambda^\bullet \g$, we see $\chi(\Omega)$ is a nilpotent operator by Theorem~\ref{Theorem-CupProduct}.  Since $\nabla_{GM}$ is a perturbation of $\widetilde{\nabla}$ by a nilpotent $\widetilde{\nabla}$-parallel operator, the result follows from \cite[Proposition 4.12]{Yashinski-SmoothDefs}.
\end{proof}

\section{Integrating $\widetilde{\nabla}$ for noncommutative tori}

%
%
%

In this section, we specialize to the noncommutative tori deformation $\{ \mathcal{A}_{t\Theta}\}_{t \in J}$ for a given $n \times n$ skew-symmetric real matrix $\Theta$ and $J = \R$.  Let $A_J$ denote the algebra of sections of this deformation and let $\g = \Span\{\delta_1, \ldots , \delta_n\}$, which is an abelian Lie algebra of derivations on $A_J$.  The canonical connection $\nabla = \frac{d}{dt}$ is $\g$-invariant, and
\[ E = \delta \nabla = \frac{1}{2\pi i} \sum_{j>k} \theta_{jk} \cdot \delta_j \smile \delta_k. \]

As in the previous section, we can define the connection
\[ \widetilde{\nabla} = L_\nabla + \frac{1}{2\pi i} \sum_{j > k} \theta_{jk}\cdot L\{ \delta_j, \delta_k \} \] on the $\g$-invariant complex $C_\bullet^{\g}(A_J)$ which descends to a connection on $HP_\bullet^{\g}(A_J)$.  By the results of the previous section, $\widetilde{\nabla}$ is a nilpotent perturbation of $\nabla_{GM}$ on $HP_\bullet^{\g}(A_J)$, and so $\widetilde{\nabla}$ is integrable if and only if $\nabla_{GM}$ is integrable.

We shall now show that we are not losing anything in passing to $\g$-invariant cyclic homology, in that the canonical map $HP_\bullet(\A_\Theta) \to HP_\bullet^{\g}(\A_\Theta)$ is an isomorphism.

\begin{theorem} \label{Theorem-InvariantChainEquivalence}
The canonical map $C_{\per}(\A_\Theta) \to C_{\per}^{\g}(\A_\Theta)$ induces an isomorphism $HP_\bullet(\A_\Theta) \cong HP_\bullet^{\g}(\A_\Theta)$
\end{theorem}

\begin{proof}
The space $\A_\Theta = \S(\Z^n)$ is a topological direct sum graded by $\Z^n$.  For each $1 \leq j \leq n$, the $j$-th degree of a homogenous element $a = u_1^{\alpha_1}\cdot \ldots \cdot u_n^{\alpha_n}$ is $\deg_j a = \alpha_j$.  The complex $C_{\per}(\A_\Theta)$ is also graded by $\Z^n$.  If $\omega = a_0 \otimes \ldots \otimes a_m$ is an elementary tensor of homogeneous elements $a_i \in \A_\Theta$, then define
\[ \deg_j \omega = \sum_{i=0}^m \deg_j a_i. \]
The Lie derivative $L_{\delta_j}$ has the property that
\[ L_{\delta_j}\omega = (2\pi i \cdot \deg_j \omega)\omega \]
for all homogeneous chains.  From this, we see that $\g \cdot C_{\per}(\A_\Theta)$ is a closed direct summand of $C_{\per}(\A_\Theta)$ which is complemented by $\bigcap_{j=1}^n \ker L_{\delta_j}$.  This complement is just the subspace of homogeneous elements whose $\Z^n$-grading is $(0, 0, \ldots, 0)$.  By the Cartan homotopy formula (Theorem \ref{Theorem-CHF}), $L_{\delta_j} = 0$ as an operator on $HP_\bullet(\A_\Theta)$.  It follows that the summand $\g\cdot C_{\per}(\A_\Theta)$ is acyclic, and so the quotient map
\[ C_{\per}(\A_\Theta) \to C_{\per}^\g(\A_\Theta) \]
is a quasi-isomorphism.
\end{proof}

The above proof can be carried out because the action of $\g$ by derivations on $\A_\Theta$ is the infinitesimal of an action of the Lie group $\T^n$ by algebra automorphisms.  Thus $\A_\Theta$ decomposes as a topological direct sum of eigenspaces indexed by $\Z^n$ for the action of $\g$.
Note that the same proof shows that $C_{\per}(A_J) \to C_{\per}^\g(A_J)$ is a quasi-isomorphism for the algebra of sections $A_J$ of the deformation.  Here, the homology theories are considered over the ground ring $C^\infty(J)$.

\begin{theorem} \label{Theorem-TildeNablaIsIntegrable}
For the algebra of sections $A_J$ of the noncommutative tori deformation, the connection $\widetilde{\nabla}$ is integrable on $C_{\per}(A_J)$ and consequently on $C_{\per}^{\g}(A_J)$.
\end{theorem}

\begin{proof}
It suffices to prove $\widetilde{\nabla}$ is integrable on $C_{\per}(A_J)$.  The quotient $C_{\per}^\g(A_J)$ is obtained from $C_{\per}(A_J)$ as a quotient by a closed $\widetilde{\nabla}$-invariant direct summand, so integrability of $\widetilde{\nabla}$ on $C_{\per}^\g(A_J)$ will follow immediately.

Since $\widetilde{\nabla}$ restricts to a connection on $C_m(A_J)$ for each $m$, it suffices to prove $\widetilde{\nabla}$ is integrable on $C_m(A_J)$.  
Given $m+1$ multi-indices $\alpha^0, \ldots , \alpha^m \in \Z^n$, we shall use the notation
\[ u^{\overline{\alpha}} = u^{\alpha^0} \otimes u^{\alpha^1} \otimes \ldots \otimes u^{\alpha^m} \in C_m(\A_\Theta). \]  An element $\omega \in C_m(\A_\Theta)$ has the form
\[ \omega = \sum_{\overline{\alpha}} c_{\overline{\alpha}} u^{\overline{\alpha}}. \]
Since $\S(\Z^n)^{\potimes (m+1)} \cong \S(\Z^{n(m+1)})$, the coefficients $c_{\overline{\alpha}}$ must be of rapid decay.  Notice that $L\{\delta_j, \delta_k\}$ is diagonal with respect to the basis $\{u^{\overline{\alpha}}\}$, and
\[ \frac{1}{2\pi i}L\{\delta_j, \delta_k\} u^{\overline{\alpha}} = 2\pi i\cdot R(\overline{\alpha}) u^{\overline{\alpha}} \] for some real-valued polynomial $R(\overline{\alpha})$ in the multi-indices of $\overline{\alpha}$.  An element $\sigma = \sum_{\overline{\alpha}} f_{\overline{\alpha}}u^{\overline{\alpha}} \in C_m(A_J)$ with $f_{\overline{\alpha}} \in C^\infty(J)$ satisfies $\widetilde{\nabla}(\sigma) = 0$ if and only if
\[ f_{\overline{\alpha}}' + 2\pi i\cdot R(\overline{\alpha})f_{\overline{\alpha}} = 0 \]
for each $\overline{\alpha}$.  Given an initial condition $f_{\overline{\alpha}}(s) = c_{\overline{\alpha}}$, the unique solution is given by
\[ f_{\overline{\alpha}}(t) = c_{\overline{\alpha}}\exp(-2\pi i\cdot R(\overline{\alpha})(t-s)). \]  Notice that $|f_{\overline{\alpha}}(t)| = |c_{\overline{\alpha}}|$.  So for each $t \in J$, the coefficients $\{f_{\overline{\alpha}}(t)\}$ satisfy the same decay conditions as $\{c_{\overline{\alpha}}\}.$  Therefore if $\omega = \sum_{\overline{\alpha}} c_{\overline{\alpha}} u^{\overline{\alpha}} \in C_m(\A_{s\Theta})$, then $\rho(t) = \sum_{\overline{\alpha}}f_{\overline{\alpha}}(t)u^{\overline{\alpha}} \in C_m(\A_{t\Theta})$ for each $t \in J$.  Moreover, since $R(\overline{\alpha})$ is a polynomial function, it follows that $\rho$ is smooth as a function of $t$.  The solution $\rho$ is also smooth as we vary the initial parameter $s$, so $\widetilde{\nabla}$ is integrable by Theorem \ref{Theorem-IntegrabilityCriteria}.

\end{proof}


\begin{corollary}
For any $n \times n$ skew-symmetric matrix $\Theta$, there is a parallel translation isomorphism
\[ HP_{\bullet}(C^\infty(\T^n)) \cong HP_{\bullet}(\A_\Theta). \]
Consequently, 
\[ HP_0(\A_\Theta) \cong \C^{2^{n-1}}, \qquad HP_1(\A_\Theta) \cong \C^{2^{n-1}}. \]
\end{corollary}

%

\begin{proof}
Since $\widetilde{\nabla}$ commutes with $b+B$ on $C_{\per}^{\g}(A_J)$, its parallel translation operators are isomorphisms of chain complexes.  Thus we have
\[ \xymatrix{
HP_\bullet(C^\infty(\T^n)) \ar[r]^-\cong & HP_\bullet^{\g}(C^\infty(\T^n)) \ar[r]^-{P^{\widetilde{\nabla}}_{01}} & HP_\bullet^{\g}(\A_\Theta) \ar[r]^-\cong & HP_\bullet(\A_\Theta).
} \]

As shown in \cite{MR823176}, if $M$ is a compact smooth manifold, then
\[ HP_\bullet(C^\infty(M)) \cong \bigoplus_{k} H_{dR}^{\bullet + 2k}(M,\C), \]
where $H_{dR}^\bullet(M,\C)$ is the complex-valued de Rham cohomology of $M$.  Now, $H_{dR}^m(\T^n, \C)$ is a vector space of dimension $\binom{n}{m},$ and this gives the result.
\end{proof}

\begin{corollary}
For the algebra of sections $A_J$ of the noncommutative tori deformation, the Gauss-Manin connection is integrable on $HP_\bullet(A_J)$.
\end{corollary}

\begin{proof}
This is immediate from Theorem \ref{GMiffTildeNabla}, but we don't actually need the nilpotence condition to prove this.  Using \cite[Proposition 5.7]{Yashinski-SmoothDefs}, we have an isomorphism of $C^\infty(J)$-modules
\[ HP_\bullet(A_J) \cong HP_\bullet^{\g}(A_J) \cong C^\infty(J, HP_\bullet^{\g}(\A_0)) \cong C^\infty(J, HP_\bullet(C^\infty(\T^n)) \]
because $\widetilde{\nabla}$ is integrable on the complex $C_{\per}^\g(A_J)$.
Here we see that $\nabla_{GM}$ is a connection on a finite rank trivial bundle, so it must be integrable.
\end{proof}

Analogous results can be obtained for periodic cyclic cohomology by duality.  For example, we can consider
\[ C_{\g}^{\per}(\A_\Theta) = C_{\per}^{\g}(\A_\Theta)^* = \{ \phi \in C^{\per}(\A_\Theta) ~:~ \phi(L_{\delta_j} \omega) = 0, ~\forall \omega \in C_{\per}(\A_\Theta), ~\forall j \}. \]
This is the space of all cochains which are supported on chains whose $\Z^n$-grading is $(0,0,\ldots , 0)$.  The inclusion $C^{\per}_\g(\A_\Theta) \to C^{\per}(\A_\Theta)$ is the transpose of the quotient map from Theorem \ref{Theorem-InvariantChainEquivalence}, and also is a quasi-isomorphism.  Since the underlying space of each fiber of $C_{\per}^{\g}(A_J)$ is a nuclear Fr\'{e}chet space, it follows that the dual connection $\widetilde{\nabla}^\dual$ is automatically integrable on $C^{\per}_\g(A_J) = C_{\per}^\g(A_J)^\dual$ \cite{Yashinski-SmoothDefs}.  So we have parallel translation isomorphisms $HP^\bullet(C^\infty(\T^n)) \cong HP^\bullet(\A_\Theta)$ for periodic cyclic cohomology as well.

We have proved the rigidity of periodic cyclic homology/cohomology for the deformation of noncommutative tori.  It is interesting to note that the Hochschild homology/cohomology and (non periodic) cyclic homology/cohomology are very far from rigid in this deformation.  As an example, $HH^0(A) = HC^0(A)$ is the space of all traces on the algebra $A$.  Now in the simplest case where $n=2$ and \[ \Theta = \begin{pmatrix} 0 & -\theta \\ \theta & 0 \end{pmatrix},\] it is well-known that there is a unique (normalized) trace on $\A_\Theta$ when $\theta \notin \Q$ and an infinite dimensional space of traces when $\theta \in \Q$.  For example, every linear functional on the commutative algebra $\A_0 \cong C^\infty(\T^n)$ is a trace, and thus $HH^0(C^\infty(\T^n)) = C^\infty(\T^n)^*$ is the space of distributions on $\T^n$.  Moreover, Connes showed in \cite{MR823176} that in the case $\theta \notin \Q$, $HH^1(\A_\Theta)$ and $HH^2(\A_\Theta)$ are either finite dimensional, or infinite dimensional and non-Hausdorff depending on the diophantine properties of $\theta$.  Looking back, we conclude that there are no integrable connections on $C^{\bullet}(A)$ that commute with $b$, as such a connection would imply rigidity of Hochschild cohomology.

However, our connection $\widetilde{\nabla}$ does commute with $b$ on the invariant complex $C^\bullet_{\g}(A)$.  This shows that the invariant Hochschild cohomology $HH^\bullet_{\g}(\A_\Theta)$ is independent of $\Theta$.  For example, there is exactly one (normalized) $\g$-invariant trace on $C^\infty(\T^n)$, and that corresponds to integration with respect to the only (normalized) translation invariant measure.  Thus $HH^0_{\g}(C^\infty(\T^n)) = HH^0_{\g}(\A_\Theta) = \C$.  Consequently, the canonical map $HH^\bullet_{\g}(\A_\Theta) \to HH^\bullet(\A_\Theta)$ is not, in general, an isomorphism.

The argument presented here should work with other variants of cyclic cohomology.  One example is Connes' entire cyclic cohomology which is constructed by allowing for infinite cochains $(\phi_n) \in \prod_n C^n(A)$ satisfying a certain growth condition \cite{MR953915}.  The Lie derivative and contraction operators extend to the entire cochain complex.  One can introduce the connection $\widetilde{\nabla}$ on the $\g$-invariant entire cyclic cochain complex, and it is likely integrable, though we haven't checked the analytic details.

\section{Calculations with $\nabla^{GM}$ and the Chern-Connes pairing}

We have shown that $\nabla_{GM}$ is integrable for noncommutative tori in a rather indirect way by proving integrability of the auxiliary connection $\widetilde{\nabla}$.  However, it is still useful to understand $\nabla_{GM}$ here because it is canonical and has good properties with respect to the Chern-Connes pairing.  Here, we shall calculate with the dual connection $\nabla^{GM}$ to determine the parallel translation maps, as well as the deformation of the Chern-Connes pairing.

\subsection{Cyclic cocycles, characteristic maps, and cup products}

By a \emph{cyclic $k$-cocycle} we mean an element $\phi \in C^k(A)$ such that
\[ b\phi = 0, \qquad \phi(e, a_1, \ldots , a_k) = 0, \qquad \phi(a_k, a_0, \ldots , a_{k-1}) = (-1)^k \phi(a_0, a_1, \ldots , a_k). \]
A cyclic cocycle $\phi$ automatically satisfies $B\phi = 0$, and so $\phi$ gives a cohomology class in $HP^\bullet(A)$.  Recall that elements in the first slot of a cochain can be in the unitization $A_+$.  Below we will use the notation $\widetilde{a}_0$ for elements of the unitization $A_+$ and just $a_0$ for elements of $A$.

We return to the general setting of section \ref{Section-gInvariantComplex}.
Suppose that $\g$ is an abelian Lie algebra of derivations on an algebra $A$.  In addition, suppose $\tau$ is a trace on $A$ which is $\g$-invariant in the sense that
\[ \tau \circ X = 0, \qquad \forall X \in \g. \]  Notice $\tau$ is a cyclic $0$-cocycle, and the $\g$-invariance implies $[\tau] \in HP^0_{\g}(A)$.  Define the \emph{characteristic map} $\gamma: \Lambda^\bullet \g \to C^\bullet(A)$ by
\[ \gamma(X_1 \wedge \ldots \wedge X_k)(\widetilde{a}_0, \ldots , a_k) = \frac{1}{k!}\sum_{\sigma \in \mathbb{S}_k} (-1)^\sigma \tau(\widetilde{a}_0 X_{\sigma(1)}(a_1) X_{\sigma(2)}(a_2) \ldots X_{\sigma(k)}(a_k)). \]

\begin{proposition}
The functional $\gamma(X_1 \wedge \ldots \wedge X_k)$ is a $\g$-invariant cyclic $k$-cocycle.
\end{proposition}

\begin{remark}
The map $\gamma$ is a simple case of the Connes-Moscovici characteristic map in Hopf cyclic cohomology \cite{MR1657389}.  In their work, $\H$ is a Hopf algebra equipped with some extra structure called a modular pair, and $A$ is an algebra equipped with a Hopf action of $\H$.  Assuming $A$ possesses a trace that is compatible with the modular pair, they construct a map
\[ \gamma: HP^\bullet_{Hopf}(\H) \to HP^\bullet(A) \] from the Hopf periodic cyclic cohomology of $\H$ to the ordinary periodic cyclic cohomology of $A$.  In our situation, $\H = \U(\g)$, the universal enveloping algebra of $\g$.  The fact that $\g$ acts on $A$ by derivations implies that the action of $\U(\g)$ on $A$ is a Hopf action.  The compatibility condition for the trace follows from the fact that our trace is $\g$-invariant.  As was shown in \cite{MR1657389},
\[ HP_{Hopf}^\bullet(\U(\g)) \cong \bigoplus_{k = \bullet \text{ mod } 2} H_k^{Lie}(\g, \C), \]
where $H_k^{Lie}(\g, \C)$ is the Lie algebra homology of $\g$ with coefficients in the trivial $\g$-module $\C$.  As $\g$ is abelian, there is an isomorphism
\[ H_k^{Lie}(\g, \C) \cong \Lambda^k(\g). \]
The obtained characteristic map
\[ \gamma: \Lambda^\bullet(\g) \to HP^\bullet(A) \]
is the map defined above.
\end{remark}

The fact that $\gamma(X_1 \wedge \ldots \wedge X_k)$ is $\g$-invariant relies on the fact that $\g$ is abelian.  In this case, the characteristic map factors through the inclusion $HP^\bullet_\g(A) \to HP^\bullet(A),$ and we obtain a characteristic map
\[ \gamma: \Lambda^\bullet(\g) \to HP^\bullet_{\g}(A). \]

\begin{lemma}
Let $X_1, \ldots , X_n$ be derivations on an algebra $A$, and let $\tau$ be a trace on $A$.  There exists $\psi \in C^{n-1}(A)$ such that $B\psi = 0$ and
\begin{align*}
&\tau(\widetilde{a}_0X_1(a_1)\ldots X_n(a_n))\\
&\qquad = \frac{1}{n}\sum_{j=1}^n(-1)^{(j-1)(n+1)}\tau(\widetilde{a}_0 X_j(a_1) \ldots X_n(a_{n-j+1})X_1(a_{n-j+2})\ldots X_{j-1}(a_n))\\
&\qquad \qquad + (b\psi)(\widetilde{a}_0, \ldots , a_n).
\end{align*}
\end{lemma}

\begin{proof}
Given any $n$ derivations $Y_1, \ldots , Y_n$, the cochain $\phi \in C^{n-1}(A)$ given by
\[ \phi(a_0, \ldots , a_{n-1}) = \tau(Y_1(a_0)Y_2(a_1)\ldots Y_n(a_{n-1})), \qquad \phi(e, a_1, \ldots , a_{n-1}) = 0,\] satisfies
\[ (b\phi)(\widetilde{a}_0, \ldots , a_n) = \tau(\widetilde{a}_0Y_1(a_1) \ldots Y_n(a_n) + (-1)^n \widetilde{a}_0Y_2(a_1)\ldots Y_n(a_{n-1})Y_1(a_n)) \] and $B\phi = 0$.  It follows that
\begin{align*}
&\psi(a_0, \ldots , a_{n-1})\\
&\qquad = \frac{1}{n}\sum_{j=1}^{n-1} (-1)^{(j-1)(n+1)}(n-j) \tau(X_j(a_0) X_{j+1}(a_1)\ldots X_{j-1}(a_n)),\\
&\psi(e, a_1, \ldots , a_{n-1}) = 0
\end{align*}
satisfies the conclusions of the lemma.
\end{proof}

Recall that for any $Z \in \g$, the contraction $I_Z$ is a chain map on the invariant complex $C_{\per}^\g(A)$.  We shall use the same notation $I_Z$ to denote its transpose, which is a chain map on $C^{\per}_\g(A)$.  We now compute this operator on the image of the characteristic map.

\begin{proposition}
For any $Z \in \g$ and $\omega \in \Lambda^\bullet \g$,
\[ I_Z[\gamma(\omega)] = [\gamma(Z \wedge \omega)] \] in $HP^\bullet_\g(A)$.
\end{proposition}

\begin{proof}
Let $\phi = \gamma(X_1 \wedge \ldots \wedge X_k)$.  Since $\phi(e, a_1, \ldots , a_k) = 0$, we immediately have $S_Z \phi = 0$.  Thus, $I_Z \phi = \iota_Z \phi$, and
\begin{align*}
(\iota_Z \phi)(a_0, \ldots , a_{k+1}) &= \phi(a_0Z(a_1), a_2, \ldots a_{k+1})\\
&= \frac{1}{k!}\sum_{\sigma \in \mathbb{S}_k} (-1)^\sigma \tau(a_0 Z(a_1)X_{\sigma(1)}(a_2) X_{\sigma(2)}(a_3) \ldots X_{\sigma(k)}(a_{k+1}))\\
&= \gamma(Z \wedge X_1 \wedge \ldots \wedge X_k)(a_0, \ldots , a_{k+1}) + (b\psi)(a_0, \ldots , a_{k+1})
\end{align*} for some $\psi$ with $B\psi = 0$ by applying the previous lemma to each term in the sum.  Hence, $I_Z \gamma(X_1 \wedge \ldots \wedge X_k) = \gamma(Z \wedge X_1 \wedge \ldots \wedge X_k) + (b+B)\psi$.
\end{proof}

As in the homology case (Theorem~\ref{Theorem-CupProduct},) there is an algebra map
\[ \chi: \Lambda^\bullet(\g) \to \End(HP^\bullet_{\g}(A)) \] given by
\[ \chi(X_1 \wedge \ldots \wedge X_k) = I_{X_1} I_{X_2} \ldots I_{X_k}. \]

\begin{corollary} \label{Corollary-CupProductGeneralizesCharacteristicMap}
For any $\omega \in \Lambda^\bullet \g$,
\[ [\gamma(\omega)] = \chi(\omega)[\tau] \] in $HP^\bullet_{\g}(A)$.
\end{corollary}

\begin{remark}
A generalization of the Connes-Moscovici characteristic map was constructed in \cite{MR2112033}.  A special case of this construction is a cup product
\[ \smile: HP^p_{Hopf}(\H) \otimes  HP^q_{\H}(A) \to HP^{p+q}(A), \] where $HP^\bullet_{\H}(A)$ is the periodic cyclic cohomology of $A$ built out of cochains which are invariant in some sense with respect to an action of $\H$.  In the Connes-Moscovici picture, the properties of the trace $\tau$ ensures that it gives a cohomology class in $HP^\bullet_{\H}(A)$, and
\[ [\omega] \smile [\tau] = \gamma[\omega] \] for all $[\omega] \in HP^\bullet_{Hopf}(\H)$.
In our situation where $\H = \U(\g)$, we have that $HP^\bullet_{\H}(A) = HP^\bullet_{\g}(A)$ and the cup product is a map
\[ \smile: \Lambda^p \g \otimes HP^q_{\g}(A) \to HP^{p+q}(A). \]
Our map $\chi: \Lambda^\bullet(A) \to \End(HP^\bullet_\g(A))$ followed by the canonical inclusion $HP^\bullet_{\g}(A) \to HP^\bullet(A)$ coincides with this cup product.  The fact that the cup product sends $\g$-invariant cocycles to $\g$-invariant cocycles is a consequence of the fact that $\g$ is abelian.
\end{remark}

\subsection{$\nabla^{GM}$-derivatives of characteristic cocycles}
Now consider the situation of section \ref{Subsection-ConnectionsOngInvariantComplex}.  Suppose $A_J$ is the algebra of sections of a deformation, $\g$ is an abelian Lie algebra of derivations on $A_J$, $\nabla$ is a $\g$-invariant connection on $A_J$ satisfying
\[ \nabla(a_1 a_2) = \nabla(a_1)a_2 + a_1\nabla(a_2) + \sum_{i=1}^r X_i(a_1)Y_i(a_2),\qquad a_1,a_2 \in A_J \]
for $X_i, Y_i \in \g$.  In addition we assume $A_J$ has a $\g$-invariant parallel trace $\tau$.  That is, there is a trace $\tau: A_J \to C^\infty(J)$ such that $\tau \circ Z = 0$ for all $Z \in \g$ and $\tau \circ \nabla = \frac{d}{dt} \circ \tau$.  Then there is a characteristic map
\[ \gamma: \Lambda^\bullet \g \to HP^\bullet_{\g}(A_J) \] as in the previous section, where the cohomology is considered over $C^\infty(J)$.  The dual of the connection
\[ \widetilde{\nabla} = L_\nabla + \sum_{i=1}^r L\{X_i,Y_i\} \] is given by
\[ \widetilde{\nabla}^\dual \phi = \frac{d}{dt} \circ \phi - \phi \circ L_\nabla - \sum_{i=1}^r \phi \circ L\{X_i,Y_i\} \] on $C^{\per}_{\g}(A_J)$.

\begin{proposition} \label{Proposition-CanonicalCyclicCocyclesAreParallel}
For any $\omega \in \Lambda^\bullet \g$, we have $\widetilde{\nabla}^\dual \gamma(\omega) = 0.$
\end{proposition}

\begin{proof}
Letting $\phi = \gamma(Z_1 \wedge \ldots \wedge Z_m)$, one can show that
\[ \frac{d}{dt} \phi = \phi \circ L_\nabla + \sum_{i=1}^r \phi \circ L\{X_i,Y_i\} \] by using $\frac{d}{dt} \circ \tau = \tau \circ \nabla$ and the identity
\begin{align*}
\nabla(a_0 \ldots a_m) &= \sum_{j=0}^k a_0\ldots \nabla(a_j) \ldots a_m\\
&\qquad + \sum_{i=1}^r \sum_{j < k} a_0\ldots X_i(a_j) \ldots Y_i(a_k)\ldots a_m
\end{align*} The rest follows from the fact that $\g$ is abelian and $\nabla$ is $\g$-invariant.
\end{proof}

As before, let $\Omega = \sum_{i=1}^r X_i \wedge Y_i \in \Lambda^2 \g$.  Recall from Proposition \ref{Proposition-NablaGMIsNilpotentPerturbation} that
\[ \nabla_{GM} = \widetilde{\nabla} + \chi(\Omega) \]
as operators on $HP_\bullet^\g(A_J)$.  Dualizing gives
\[ \nabla^{GM}[\phi] = \widetilde{\nabla}^\dual [\phi] - \sum_{i=1}^r [\phi \circ \chi(X_i \wedge Y_i)] \]
and further
\[ [\phi \circ \chi(X_i \wedge Y_i)] = [\phi \circ (I_{X_i} I_{Y_i})] = I_{Y_i} I_{X_i} [\phi] = - \chi (X_i \wedge Y_i) [\phi]. \]  Consequently,
\[ \nabla^{GM} = \widetilde{\nabla}^\dual + \chi(\Omega)\]
as operators on $HP^\bullet_\g(A_J)$.  Combining the previous results, we immediately obtain the following.

\begin{theorem} \label{Theorem-NablaGMDerivative}
In the above situation, for any $\omega \in \Lambda^\bullet \g$,
\[ \nabla^{GM}[\gamma(\omega)] = [ \gamma(\Omega \wedge \omega)] \] in $HP^\bullet_\g(A_J)$.
\end{theorem}

It is worth mentioning that this result is independent of the integrability of $\widetilde{\nabla}^\dual$ or $\nabla^{GM}$.  Using this theorem, we can explicitly describe $\nabla^{GM}$-parallel classes through a given characteristic cocycle.

\begin{corollary} \label{Corollary-NablaGMParallelSection}
Let $\omega \in \Lambda^\bullet \g$ and view $\gamma(\omega) \in HP^\bullet_\g(A_s)$.  Then the cocycle $\phi \in HP^\bullet_{\g}(A_J)$ given by
\[ \phi = \sum_{p=0}^{\lfloor \dim \g /2 \rfloor} (-1)^p \frac{(t-s)^p}{p!} \gamma(\Omega^{\wedge p} \wedge \omega) \]
is a $\nabla^{GM}$-parallel section through $\gamma(\omega) \in HP^\bullet_\g(A_s)$.
\end{corollary}

Also from Theorem \ref{Theorem-NablaGMDerivative}, we obtain the following result about the deformation of the Chern-Connes character.

\begin{corollary} \label{Corollary-DerivativeOfChernConnesPairing}
Let $\omega \in \Lambda^\bullet \g$.
\begin{enumerate}[(i)]
\item For any idempotent $P \in M_N(A_J)$,
\[ \frac{d}{dt} \langle [\gamma(\omega)], [P] \rangle = \langle  [\gamma(\Omega \wedge \omega)], [P] \rangle. \]
\item For any invertible $U \in M_N(A_J)$,
\[ \frac{d}{dt} \langle [\gamma(\omega)], [U] \rangle = \langle  [\gamma(\Omega \wedge \omega)], [U] \rangle. \]
\end{enumerate}
\end{corollary}

We can be a little more explicit with these formulas using the form of the cycles $\ch P$ and $\ch U$.  First notice that $M_N(A_J) \cong M_N(\C) \otimes A_J$ is the algebra of sections of a deformation that has the same properties as $A_J$.  Namely we have an action of an abelian Lie algebra $\g$ generated by the derivations $\id \otimes X_i$ and $\id \otimes Y_i$.  There is a $\g$-invariant connection $\id \otimes \nabla$ as well as a parallel $\g$-invariant trace $\tr \otimes \tau$.  Here $\tr: M_N(\C)$ is the usual matrix trace.  Notationally, we will refer to the above objects as simply $X_i, Y_i, \nabla$, and $\tau$ respectively.  The connection $\nabla$ satisfies
\[ \nabla(a_1a_2) = \nabla(a_1)a_2 + a_1\nabla(a_2) + \sum_{i=1}^r X_i(a_1)Y_i(a_2),\qquad a_1, a_2 \in M_N(A_J), \]
so all the previous results apply to this deformation.  Notice we have a commutative diagram
\[ \xymatrix{
\Lambda^\bullet \g \ar[r]^-{\gamma} \ar[rd]^-{\gamma} & C^{\per}(A_J) \ar[d]^-{T^\dual}\\
& C^{\per}(M_N(A_J))
} \]
where $T^\dual$ is the transpose of the generalized trace $T: C_{\per}(M_N(A_J)) \to C_{\per}(A_J)$.  Given a projection $M_N(A_J)$ and a cocycle $\phi \in C^{\per}_{C^\infty(J)}(A_J)$, we will view the Chern-Connes pairing as a pairing between a cocycle and a cycle on $M_N(A_J)$:
\[ \langle [\phi], [P] \rangle = \langle T^\dual \phi, \ch P \rangle. \]
So if $\omega \in \Lambda^{2m} \g$, then by using the explicit form of $\ch P \in C_{\even}^{C^\infty(J)}(M_N(A_J))$,
\[ \langle [\gamma(\omega)], [P] \rangle = (-1)^m\frac{(2m)!}{m!}\gamma(\omega)(P, P, \ldots , P), \] where $\gamma(\omega)$ is viewed as a cocycle in $C^{\even}_{C^\infty(J)}(M_N(A_J)),$ and there are $2m+1$ appearances of $P$ on the right hand side.  Similarly, if $U \in M_N(A_J)$ is an invertible and $\omega \in \Lambda^{2m+1}\g$, then
\[ \langle [\gamma(\omega)], [U] \rangle = (-1)^m m! \gamma(\omega)(U^{-1}, U, \ldots , U^{-1}, U). \]  Combining these with Corollary \ref{Corollary-DerivativeOfChernConnesPairing}, we obtain the following result.

\begin{theorem} \label{Theorem-ExplicitDerivativeOfChernConnesPairing}
In the above situation,
\begin{enumerate}[(i)]
\item If $\omega \in \Lambda^{2m}\g$, and $P \in M_N(A_J)$ is an idempotent,
\[ \frac{d}{dt} \gamma(\omega)(P, P, \ldots , P) = -(4m+2) \gamma(\Omega \wedge \omega)(P, P, \ldots , P). \]
\item If $\omega \in \Lambda^{2m+1}\g$, and $U \in M_N(A_J)$ is invertible,
\[ \frac{d}{dt} \gamma(\omega)(U^{-1}, U, \ldots , U^{-1}, U) = -(m+1) \gamma(\Omega \wedge \omega)(U^{-1}, U, \ldots , U^{-1}, U). \]
\end{enumerate}
\end{theorem}

\subsection{Noncommutative tori}

Here, we shall apply our results to the noncommutative tori deformation.

\begin{theorem}
For every $\Theta$, the map $\gamma: \Lambda^\bullet \g \to HP^\bullet_\g(A_\Theta)$ is an isomorphism of $\Z/2$-graded spaces.
\end{theorem}

\begin{proof}
By Theorem~\ref{Theorem-TildeNablaIsIntegrable} and Proposition~\ref{Proposition-CanonicalCyclicCocyclesAreParallel}, is suffices to prove this for $\A_0 \cong C^\infty(\T^n)$. Let $s_1, \ldots , s_n$ be the coordinates in $\T^n = \R^n/\Z^n$.  Choosing a subset of coordinates $s_{i_1}, \ldots , s_{i_m}$ determines a subtorus $T$ of dimension $m$.  All such subtori are in bijection with a generating set of homology classes of $\T^n$.  The $\g$-invariant de Rham cycle corresponding to $T$ is given by taking the average of the integral of a differential form over all subtori parallel to $T$.  The cochain in $C^m(C^\infty(\T^n))$ corresponding to this cycle is
\[ \phi_T(f_0, \ldots , f_m) = \int_{\T^n/T} \left( \int_{gT} f_0 df_1 \wedge \ldots \wedge df_m\right) dgT, \] and one can show
$\phi_T = \gamma(\delta_{i_1} \wedge \ldots \wedge \delta_{i_m})$ up to a scalar multiple.
\end{proof}

Hence the characteristic cocycles form a basis for $HP^\bullet(\A_\Theta)$.  So we can use Corollary \ref{Corollary-NablaGMParallelSection} to explicitly describe the parallel translation of $\nabla^{GM}$.  Let's do this for the $n=2$ case.  Here, the noncommutative torus is determined by a single real parameter $\theta := \theta_{21}$, and we shall denote the algebra by $\mathcal{A}_\theta$.  In the case $\theta \notin \Q$, $\mathcal{A}_\theta$ is also known as (the smooth version of) the irrational rotation algebra.  We shall consider $\{\mathcal{A}_\theta\}_{\theta \in J}$ as a smooth one-parameter deformation, where $J \subset \R$ is an open interval containing $0$.  Here, $\g = \Span\{\delta_1, \delta_2\}$ and $\Omega = \frac{1}{2\pi i} \delta_2 \wedge \delta_1$.  Let $\tau_2$ be the cyclic $2$-cocycle $\tau_2 = \frac{1}{\pi i}\gamma(\delta_1 \wedge \delta_2)$, which is given explicitly by
\[ \tau_2(a_0, a_1, a_2) = \frac{1}{2\pi i}\tau(a_0 \delta_1(a_1) \delta_2(a_2) - a_0 \delta_2(a_1) \delta_1(a_2)). \]
Then from Theorem \ref{Theorem-NablaGMDerivative},
\[ \nabla^{GM}[\tau_2] = \frac{1}{\pi i (2\pi i)}[\gamma(\delta_2 \wedge \delta_1 \wedge \delta_1 \wedge \delta_2] = 0, \]
so that $[\tau_2] \in HP^\bullet(A_J)$ is a $\nabla^{GM}$-parallel section.  Let's consider the trace $[\tau] = [\gamma(1)]$ over $\A_0$.  From Corollary \ref{Corollary-NablaGMParallelSection},
\[ [\gamma(1)] - \frac{\theta}{2\pi i}[\gamma(\delta_2 \wedge \delta_1)] = [\tau] + \frac{\theta}{2} [\tau_2] \]
is a $\nabla^{GM}$-parallel section through $[\tau]$ at the fiber $\theta = 0$.

Now consider the odd classes
\[ \tau_1^1 = \frac{1}{2\pi i}\gamma(\delta_1),\qquad \tau_1^2 = \frac{1}{2\pi i}\gamma(\delta_2). \]
Then we have
\[ \nabla^{GM}[\tau_1^j] = \frac{1}{(2\pi i)^2}[\gamma(\delta_2 \wedge \delta_1 \wedge \delta_j)] = 0 \]
for $j=1,2$.  So we can completely describe the parallel translation of $\nabla^{GM}$.

\begin{theorem}
Let $\theta \in \R$, then $P^{\nabla^{GM}}_{0, \theta}: HP^\bullet(\A_0) \to HP^\bullet(\A_\theta)$ is given by
\[ P^{\nabla^{GM}}_{0, \theta}: [\tau] \mapsto [\tau] + \frac{\theta}{2}[\tau_2], \qquad [\tau_2] \mapsto [\tau_2],\qquad [\tau_1^j] \mapsto [\tau_1^j].\]
\end{theorem}

It is interesting to notice that this parallel translation gives a nontrivial automorphism
\[ P^{\nabla^{GM}}_{0, 1}: HP^\bullet(C^\infty(\T^n)) \to HP^\bullet(\A_1) =  HP^\bullet(C^\infty(\T^n)). \]

Now let's consider the Chern-Connes pairing.  Using Theorem \ref{Theorem-ExplicitDerivativeOfChernConnesPairing}, we see that for any idempotent $P \in M_N(A_J)$,
\[ \frac{d}{d\theta} \tau(P) = \tau_2(P,P,P) \] and
\[ \frac{d^2}{d\theta^2} \tau(P) = \frac{d}{d\theta} \tau_2(P,P,P) = 0. \] Thus 
\[ \tau(P(\theta)) = \tau(P(0)) + \tau_2(P(0),P(0),P(0)) \cdot \theta. \]
Now the idempotent $P(0) \in M_N(\mathcal{A}_0) \cong M_N(C^\infty(\T^2))$ corresponds to a smooth vector bundle over $\T^2$ and the value $\tau(P(0))$ is the dimension of this bundle.  The number $\tau_2(P(0),P(0),P(0))$ is the first Chern number of the bundle, which is an integer.  So $P$ satisfies
\[ \tau(P) = C + D\theta \] for integers $C$ and $D$.

Starting with a vector bundle represented by $P(0) \in M_N(C^\infty(\T^2))$, one can always extend it to a smooth family of idempotents $P(\theta) \in M_N(C^\infty(\A_\theta))$ for small enough $\theta$ using a functional calculus argument.  Our results imply that the trace $\tau(P(\theta))$ is determined by, and can be computed from, the characteristic classes of the vector bundle $P(0)$.  The same is true for other Chern-Connes pairings.  For example, the value $\tau_2(P(\theta), P(\theta), P(\theta))$ is constant, hence an integer.  This integrality was also explained with an index formula in \cite{MR572645}.  Similarly, the pairing of an invertible with $\tau_1^j$ is integral.

These results also show that $A_\theta$ has a $K$-theory class with trace $\theta$, at least for small enough $\theta$, and so suggest that one may be able to find an idempotent in $\A_\theta$ with trace $\theta$.  Of course it is well-known now that such idempotents exist \cite{MR623572}.  Given such an idempotent $P_\theta \in \A_\theta$, one could try to extend it to $\theta = 0$ through an idempotent $P \in A_J$.  However this is impossible because $P(0)$ would necessarily satisfy $\tau_2(P(0), P(0), P(0)) \neq 0$, and the only idempotents in $\A_0$, namely $0$ and $1$, do not.  However, one can extend $P_\theta$ to an idempotent in $A_J$ provided $J$ doesn't contain integers.  One can also find an idempotent $P \in M_2(A_J)$ of trace $1 + \theta$, but only if $J \subseteq (-1, 1)$.

Two other interesting situations are when $J = \R$ or $J = \T$.  Here, the trace of any idempotent $P \in M_N(A_J)$ must be constant and integral, and $\tau_2(P,P,P) = 0$.  If $\tau_2(P,P,P) \neq 0$, then $\tau(P)$ would be negative somewhere in the $J = \R$ case, and $\tau(P)$ wouldn't be continuous in the $J = \T$ case.

Analogous results can be worked out for higher dimensional noncommutative tori.  For the deformation $\{\A_{t\Theta}\}_{t \in J}$, the element $\Omega \in \Lambda^2 \g$ is
\[ \Omega = \frac{1}{2\pi i} \sum_{j > k} \theta_{jk} \cdot \delta_j \wedge \delta_k. \]
One can explicitly compute $\nabla^{GM}$ and its parallel translation operators using Theorem \ref{Theorem-NablaGMDerivative}.  The deformation of the Chern-Connes pairing can be determined from Theorem \ref{Theorem-ExplicitDerivativeOfChernConnesPairing}.  As functions of $t$, the pairings can be higher degree polynomials, depending on the size of $n$, see also \cite{MR731772} in the case of the canonical trace.  The most interesting case would be to choose $\Theta$ so that $\Omega$ is nondegenerate in the sense that $\Omega^{\wedge \lfloor n/2 \rfloor} \neq 0$.

\end{document}